\documentclass{article}
\usepackage{amsmath}
\usepackage{amsfonts}
\usepackage{xcolor}
\usepackage{geometry}
\newgeometry{tmargin=3cm, bmargin=3cm, lmargin=3cm, rmargin=3cm}
\newtheorem{theorem}{Theorem}

\newtheorem{corollary}[theorem]{Corollary}

\newtheorem{definition}[theorem]{Definition}

\newtheorem{lemma}[theorem]{Lemma}

\newtheorem{proposition}[theorem]{Proposition}
\newtheorem{remark}[theorem]{Remark}

\DeclareMathOperator{\dive}{div}
\newenvironment{proof}[1][Proof]{\noindent\textbf{#1.} }{\ \rule{0.5em}{0.5em}}
\begin{document}
\author{Josef Diblik, Marek Galewski, Igor Kossowski and Dumitru Motreanu, }
\title{On competing $\left( p,q\right) -$Laplacian Drichlet problem with  unbounded weight}
\date{}
\maketitle
\begin{abstract}
We investigate the existence of generalized solutions to
coercive competing system driven by the $\left( p,q\right) -$Laplacian  with unbounded perturbation  corresponding to the leading term in the differential operator and with convection depending on the gradient. Some abstract principle leading to the existence of generalized solutions is
also derived basing on the Galerkin scheme.
\end{abstract}



\section{Introduction}

In this paper we want to study the following problem with homogeneous
Dirichlet boundary condition and an unbounded term $g(u)$ in the
differential operator namely: 
\begin{equation}
\begin{cases}
-\dive\left( g(u\left( x\right) )|\nabla u\left( x\right) |^{p-2}\nabla
u\left( x\right) \right) +\dive\left( |\nabla u\left( x\right)
|^{q-2}\nabla u\left( x\right) \right) =f(x,u\left( x\right) ,\nabla u\left(
x\right) ) & \text{ in }\Omega , \\ 
u\left( x\right) =0 & \text{ on }\partial \Omega%
\end{cases}
\label{problemCL}
\end{equation}%
Here $p>q>1$. Problem (\ref{problemCL}) is considered on a bounded domain $%
\Omega \subset \mathbb{R}^{N}$ with Lipschitz boundary $\partial \Omega $
and where the convection term $f:\Omega \times \mathbb{R}\times \mathbb{R}%
^{N}\to \mathbb{R}$ is subject to a suitable growth leading to the well posedness and
coercivity of the associated operator. The weight $g$ employed above
satisfies the following assumption:$\smallskip $

(\textbf{H1) }$g:%
\mathbb{R}
\rightarrow \left[ a_{0},+\infty \right) $\textit{\ is a continuous function
with }$a_{0}>0$\textit{. }$\smallskip $

The operator appearing in the left hand side of (\ref{problemCL}) with $%
g\equiv 1$ is called the competing $\left( p,q\right) -$Laplacian and was
considered for the relevant counterpart of (\ref{problemCL}) in \cite%
{LiuLivreaMontreanuZeng} and next in several subsequent research \cite{fig1}, \cite%
{fig2}, \cite{MotreanuOpenMath}, \cite{vetro1} where it was generalized and investigated from other
perspectives and in a more general setting, like to variable exponent case,
on the Heisenberg group and also with a type of Kichrhoff terms. Related results are aslo to be found in \cite{Bu}, \cite{figRaz}, \cite{RazaniBVP} The
difficulty in considering problems like 
\begin{equation*}
\begin{cases}
-\dive\left( |\nabla u|^{p-2}\nabla u\right) +\dive\left( |\nabla
u|^{q-2}\nabla u\right) =f(x,u,\nabla u) & \text{ in }\Omega , \\ 
u=0 & \text{ on }\partial \Omega 
\end{cases}%
\end{equation*}%
for a suitable convection $f$, lies in their lack of enough monotonicity on
the left hand side, despite the relevant operator being bounded. This is
why in \cite{MotreanuOpenMath} the Author proposed an approximation scheme
based directly on the ideas of the proof of the Browder-Minty Theorem on the
existence of solutions to nonlinear equations with pseudomonotone operators.
The numerical scheme which we introduce further leadsaters in order to provide some abstract result, and which originates from \cite{MotreanuOpenMath}, leads to the existence of the classical weak solution when applied for example to the problem%
\begin{equation*}
\begin{cases}
-\dive\left( |\nabla u|^{p-2}\nabla u\right) -\dive\left( |\nabla
u|^{q-2}\nabla u\right) =f(x,u,\nabla u) & \text{ in }\Omega , \\ 
u=0 & \text{ on }\partial \Omega 
\end{cases}%
\end{equation*}%
as shown in \cite{MotreanuOpenMath}.$\bigskip $

On the other hand in \cite{MotreanuTornatore} the Authors investigate via
suitable truncation technique the following problem with convection $f$ and
with the unbounded weight satisfying the structure condition similar to this
in \textbf{(H1)} (with some degeneracy added):%
\begin{equation}
\begin{cases}
-\dive\left( \nu (x,u)|\nabla u|^{p-2}\nabla u\right) =f(x,u,\nabla u)
& \text{ in }\Omega , \\ 
u=0 & \text{ on }\partial \Omega%
\end{cases}
\label{problem with a}
\end{equation}%
In this case, the differential operator pertaining to the left hand side is
not necessarily bounded which prevents using the mentioned Browder-Minty
Theorem. Nevertheless, the assumptions on convection $f$ lead to the
conclusion that the set of weak solutions is bounded and also that there is
a $L^{\infty }$ bound on each weak solution. Then it is shown that for the
original and truncated problem solutions coincide and the solvability of (%
\ref{problem with a}) follows.$\bigskip $

Our aim in present submission is to combine the approach from \cite%
{MotreanuOpenMath} as far as the approximation scheme is considered and the
estimation and truncation techniques from \cite{MotreanuTornatore} in order
to investigate problem (\ref{problemCL}). Thus we find some constant $R>0$ such that
for 
\begin{equation}
g_{R}(t)=%
\begin{cases}
g(t) & \text{ if }t\in \lbrack -R,R] \\ 
g(R) & \text{ if }t>R%
\end{cases}
\label{def of g}
\end{equation}%
we consider the following auxiliary problem:%
\begin{equation}
\begin{cases}
-\dive\left( g_{R}( u\left( x\right) )|\nabla
u\left( x\right) |^{p-2}\nabla u\left( x\right) \right) +\dive\left(
|\nabla u\left( x\right) |^{q-2}\nabla u\left( x\right) \right) =f(x,u\left(
x\right) ,\nabla u\left( x\right) ) & \text{ in }\Omega , \\ 
u\left( x\right) =0 & \text{ on }\partial \Omega .%
\end{cases}
\label{problemCL_aux}
\end{equation}%
It is shown that bounded solutions to (\ref{problemCL}) and (\ref%
{problemCL_aux}) coincide. Thus we obtain the solvability of (\ref{problemCL}%
) in the sense of generalized solution as suggested in \cite{MotreanuOpenMath}. We
introduce also some abstract tool leading to the existence of generalized
solutions by dropping the monotonicity related assumption in a version of
the Browder-Minty Theorem while retaining the continuity and the coercivity assumptions. To the best of our knowledge such a problem has
not been considered in the literature before and may lead to further
interesting problems and observations as it can be shifted to other
settings. 
\par The convection term is expressed through the Nemytskij operator
associated with a Carath\'{e}odory function $f:\Omega \times \mathbb{R}%
\times \mathbb{R}^{N}\rightarrow \mathbb{R}$, i.e., $f(x,s,\xi )$ is
measurable in $x\in \Omega $ for all $(s,\xi )\in \mathbb{R}\times \mathbb{R}%
^{N}$ and is continuous in $(s,\xi )\in \mathbb{R}\times \mathbb{R}^{N}$ for
a.e. $x\in \Omega $. The function $f$ will be subject to appropriate growth
conditions provided below \textbf{(H2)}-\textbf{(H3)} leading to the
coercivity, boundedness and continuity of the truncated operator.$\bigskip $

We seek the solutions to problem (\ref{problemCL}) in $W_{0}^{1,p}(\Omega )$, which is a standard Sobolev space, we refere to \cite{evans} for a suitable background.
. For every real number $r>1$, we set $r^{\prime }=r/(r-1)$ (the H\"{o}lder
conjugate of $r$ ). For $1<q<p<+\infty $ we have $p^{\prime
}=p/(p-1)<q^{\prime }=q/(q-1)$. The Sobolev spaces $W_{0}^{1,p}(\Omega )$
and $W_{0}^{1,q}(\Omega )$ are endowed with the norms $\Vert \nabla (\cdot
)\Vert _{L^{p}(\Omega )}$ and $\Vert \nabla (\cdot )\Vert _{L^{q}(\Omega )}$%
, respectively, where $\Vert \cdot \Vert _{L^{r}(\Omega )}$ stands for the
usual $L^{r}$-norm. The dual spaces of $W_{0}^{1,p}(\Omega )$ and $%
W_{0}^{1,q}(\Omega )$ are denoted $W^{-1,p^{\prime }}(\Omega )$ and $%
W^{-1,q^{\prime }}(\Omega )$, respectively. We assume that $N<p$. In this
case there is a continuous embedding of $W_{0}^{1,p}(\Omega )$ into $C\left( 
\overline{\Omega }\right) $ and hence into $L^{r}(\Omega )$ for any $r>0$.
There is a constant $C_{S}$ such that 
\begin{equation}
\max_{x\in \overline{\Omega }}\left\vert u\left( x\right) \right\vert \leq
C_S \Vert \nabla u\Vert _{L^{p}(\Omega )}  \label{Sobolev Ineq}
\end{equation}%
for all $u\in W_{0}^{1,p}(\Omega )$.\texttt{\ }The (negative) $p$-Laplacian $%
-\Delta _{p}:W_{0}^{1,p}(\Omega )\rightarrow W^{-1,p^{\prime }}(\Omega )$ is
defined as follows%
\begin{equation*}
\left\langle -\Delta _{p}u,v\right\rangle =\int_{\Omega }|\nabla
u(x)|^{p-2}\nabla u(x)\cdot \nabla v(x)\,\mathrm{d}x\quad \text{ for all }%
u,v\in W_{0}^{1,p}(\Omega )
\end{equation*}%
and it is a strictly monotone, continuous operator, potential and bounded
and therefore pseudomonotone. Due to the assumption $1<q<p<+\infty $ there
is a continuous embedding $W_{0}^{1,p}(\Omega )\hookrightarrow
W_{0}^{1,q}(\Omega )$. Therefore, the term $\Delta _{q}$ in the left-hand
side of (\ref{problemCL}) is well defined on $W_{0}^{1,p}(\Omega )$. The
notation $|\Omega |$ stands for the Lebesgue measure of $\Omega $.$\bigskip $

The weak solution to (\ref{problemCL}), should it exist, is defined as
follows 

\begin{equation}
\begin{split}
\int_{\Omega }g(u\left( x\right) )|\nabla u\left( x\right) |^{p-2}\nabla
u\left( x\right) \nabla v( x)\,\mathrm{d}x-&\int_{\Omega }|\nabla
u(x)|^{q-2}\nabla u\left( x\right) \nabla v( x)\,\mathrm{d}x  \\
&=\bigskip
\int_{\Omega }f(x,u\left( x\right) ,\nabla u\left( x\right) )v(
x)\,\mathrm{d}x.
\end{split}
\label{def_weak_sol}
\end{equation}%
Such a solution however cannot be reached directly due to the mentioned lack
of monotonicity, the variational
methods are not applicable here. Nevertheless we shall investigate in what
follows the truncated version of following operator $A:W_{0}^{1,p}(\Omega
)\rightarrow W^{-1,p^{\prime }}(\Omega )$
\begin{equation}
\begin{split}
\left\langle A\left( u\right) ,v\right\rangle &=\int_{\Omega }g(
u\left( x\right) )|\nabla u\left( x\right) |^{p-2}\nabla u\left(
x\right) \nabla v(x)\,\mathrm{d}x \\
&-\bigskip 
\int_{\Omega }|\nabla u(x)|^{q-2}\nabla u\left( x\right) \nabla v(
x)\,\mathrm{d}x-\int_{\Omega }f(x,u\left( x\right) ,\nabla u\left(
x\right) )v(x)\,\mathrm{d}x.
\end{split}
\label{def_operat_A}
\end{equation}

The assumptions about the convection which we impose here are as follows and
match these employed in \cite{MotreanuOpenMath} with necessary changes due
to the presence of the weight $g$ and also due the space setting we apply:

\textbf{(H2)}\textit{\ There exist a nonnegative function }$\sigma \in
L^{r_{1}}(\Omega )$\textit{\ and constants }$b\geq 0$\textit{\ and }$c\geq 0$%
\textit{\ such that }

\begin{equation*}
|f(x,s,\xi )|\leq \sigma (x)+b|s|^{r_{2}}+c|\xi |^{p-1}\quad \text{\textit{\
for a.e.} }x\in \Omega \text{, \textit{all} }s\in \mathbb{R},\xi \in \mathbb{%
R}^{N}\text{, }
\end{equation*}%
\textit{where }$r_{1},r_{2}\geq 1$.

\textbf{(H3)} \textit{There exist constants }$c_{0}<a_{0}$, $c_{1}>0$\textit{\
and }$\alpha \in \lbrack 1,p)$\textit{\ such that }

\begin{equation*}
f(x,s,\xi )s\leq c_{0}|\xi |^{p}+c_{1}\left( |s|^{\alpha }+1\right) \text{ 
\textit{for a.e.} }x\in \Omega \text{, all }s\in \mathbb{R},\xi \in \mathbb{R%
}^{N}\text{. }
\end{equation*}

In \cite{MotreanuTornatore} it is mentioned that the function $f:\Omega
\times \mathbb{R}\times \mathbb{R}^{N}\rightarrow \mathbb{R}$ given by%
\begin{equation*}
f(x,s,\xi )=|s|^{\alpha -2}s+\frac{s}{1+s^{2}}\left( |\xi
|^{p-1}+h(x)\right) \text{ \textit{for all} }(x,s,\xi )\in \Omega \times 
\mathbb{R}\times \mathbb{R}^{N},
\end{equation*}

with a constant $\alpha \in \lbrack 1,p)$ and some $h\in L^{\infty }(\Omega )
$ satisfies conditions \textbf{(H2)}-\textbf{(H3)}. Further on we
will mention also another version of assumption (H3). We will consider the
case when $\alpha =p$\ which will imply some relation between $c_{1}$\ and $\lambda _{1}$.

\section{Auxiliary results}

In this short section we cover the material which we need in the sequel and we follow. The
first eigenvalue of $-\Delta _{p}$ is given by

\begin{equation}
\lambda _{1}:=\inf_{u\in W_{0}^{1,p}(\Omega ),u\neq 0}\frac{\int_{\Omega
}|\nabla u(x)|^{p}\,\mathrm{d}x}{\int_{\Omega }|u(x)|^{p}\,\mathrm{d}x}
\label{def_lambda1}
\end{equation}

From the above definition it follows immediately that for all $u\in
W_{0}^{1,p}(\Omega ):$%
\begin{equation}
\Vert u\Vert _{L^{p}(\Omega )}\leq \frac{1}{\lambda _{1}}\Vert \nabla u\Vert
_{L^{p}(\Omega )}  \label{Poincare inequality}
\end{equation}

Let $E$ be a a reflexive Banach space. We denote by $\langle \cdot ,\cdot
\rangle $ the duality pairing between $E$ and its dual $E^{\ast }$, by $%
\rightarrow $ the strong convergence and by $\rightharpoonup $ the weak
convergence. A map $A:E\rightarrow E^{\ast }$ is called bounded if it maps
bounded sets into bounded sets. The map $A:E\rightarrow E^{\ast }$ is said
to be coercive if%
\begin{equation*}
\lim_{\Vert u\Vert \rightarrow +\infty }\frac{\langle A(u),u\rangle }{\Vert
u\Vert }=+\infty.
\end{equation*}

The map $A:E\rightarrow E^{\ast }$ is called pseudomonotone if for each
sequence $\left\{ u_{n}\right\} \subset X$ satisfying $u_{n}\rightarrow u$
in $E$ and $\limsup_{n\rightarrow \infty }\left\langle A\left(
u_{n}\right) ,u_{n}-u\right\rangle \leq 0$, it holds%
\begin{equation*}
\langle A(v),u-v\rangle \leq \liminf_{n\rightarrow \infty }\left\langle
A\left( u_{n}\right) ,u_{n}-v\right\rangle \text{ for all }v\in E.
\end{equation*}%
The map $A:E\rightarrow $ $E^{\ast }$ is monotone if all $u,v\in E$ it holds%
\begin{equation*}
\left\langle A\left( u\right) -A\left( v\right) ,u-v\right\rangle \geq 0;
\end{equation*}%
and hemicontinuous if for all $u,v,h\in E$\ function%
\begin{equation*}
s\rightarrow \left\langle A\left( u+sv\right) ,h\right\rangle
\end{equation*}%
is continuous on $\left[ 0,1\right] $.

The monotone and hemicontinuous map is pseudomonotone. The main theorem for
pseudomonotone operators reads as follows (see \cite{moteranucARL}) and is
proved via Galerkin type approximations:

\begin{theorem}
If the mapping $A:E\rightarrow E^{\ast }$ is pseudomonotone, bounded and
coercive, then it is surjective.
\end{theorem}

In the proof of the above result the Brouwer fixed point theorem is
exploited, namely its corollary:

\begin{lemma}
\label{lemma from Brouwer}Let $X$ be a finite dimensional space with the
norm $\Vert \cdot \Vert _{X}$ and let $A:X\rightarrow X^{\star }$ be a
continuous mapping. Assume that there is a constant $R>0$ such that%
\begin{equation*}
\langle A(v),v\rangle \geq 0\text{ for all }v\in X\text{ with }\Vert v\Vert
_{X}=R\text{. }
\end{equation*}%
Then there exists $u\in X$ with $\Vert u\Vert _{X}\leq R$ satisfying $A(u)=0$%
.
\end{lemma}

\section{Estimations}

We proceed with the following lemmas that will further lead to the well
posedness and the coercivity of the operator of operator $A$. The lemma
below is taken after Lemma 2.2 \cite{MotreanuOpenMath} but we provide the proof due the different assumptions when compared to \cite{MotreanuOpenMath}:

\begin{lemma}
\label{lemmaestimationf}Under assumption \textbf{(H2)} there is a constant $C>0$ such that
estimate%
\begin{equation}
\begin{split}
\left|\int_{\Omega }f(x,u\left( x\right) ,\nabla u\left( x\right)
)v(x)\,\mathrm{d}x\right| &\leq 
\int_{\Omega }\left\vert f(x,u\left( x\right) ,\nabla u\left( x\right)
)v\left( x\right) \right\vert \mathrm{~d}x \\ 
&\leq  C\left( \left\Vert \sigma
\right\Vert _{L^{r_{1}}\left(\Omega\right)}+\left\Vert u\right\Vert
_{L^{r_{2}}\left(\Omega\right)}^{r_{2}}+\left\Vert \nabla u\right\Vert _{L^{p}\left(\Omega\right)}^{p-1}\right)
\left\Vert \nabla v\right\Vert _{L^{p}\left(\Omega\right)}%
\end{split}
\label{estim _rhs}
\end{equation}%
for all $u,v\in W_{0}^{1,p}(\Omega )$.
\end{lemma}
\begin{proof}
Via assumption \textbf{(H2)} we get 
\begin{equation}\label{lemEs1}\begin{split}\int_{\Omega}|f(x,u(x),\nabla{u}(x))v(x)|\,\mathrm{d}x &\leq 
\int_{\Omega}|\sigma(x)||v(x)|\,\mathrm{d}x + b \int_{\Omega}|u(x)|^{r_2}|v(x)|\,\mathrm{d}x \\
+ c\int_{\Omega}|\nabla{u}(x)|^{p-1}|v(x)|\,\mathrm{d}x 
\end{split}
\end{equation}
For the first term of right hand side of \eqref{lemEs1} the Sobolev inequality \eqref{Sobolev Ineq} and  the H\"older inequality yield 
\[ 
\int_{\Omega}|\sigma(x)||v(x)|\,\mathrm{d}x \leq \|v\|_{C\left(\overline{\Omega}\right)}\int_{\Omega}|\sigma(x)|\,\mathrm{d}x \leq C_S |\Omega|^{\frac{r_1 - 1}{r_1}} \|\sigma\|_{L^{r_1}\left(\Omega\right)}\|\nabla{v}\|_{L^p\left(\Omega\right)}.
\]
Next, using again the Sobolev inequality for the second summands we obtain 
\[ \int_{\Omega}|u(x)|^{r_2}|v(x)|\,\mathrm{d}x \leq \|v\|_{C\left(\overline{\Omega}\right)}\|u\|^{r_2}_{L^{r_2}\left(\Omega\right)} \leq C_S \|u\|^{r_2}_{L^{r_2}\left(\Omega\right)} \|\nabla{v}\|_{L^p\left(\Omega\right)}. \]
In the final step, by the H\"older inequality combined with the Poincar\'e inequality \eqref{Poincare inequality} we have
\[ 
\int_{\Omega}|\nabla{u}(x)|^{p-1}|v(x)|\,\mathrm{d}x \leq \|\nabla{u}\|^{p-1}_{L^p\left(\Omega\right)}\|v\|_{L^p\left(\Omega\right)} \leq \frac{1}{\lambda_1}\|\nabla{u}\|^{p-1}_{L^p\left(\Omega\right)}\|\nabla{v}\|_{L^p\left(\Omega\right)}.
\]
Combing the above estimations we obtain the assertion.
\end{proof}
\begin{corollary}
\label{corollaty_Niemytskij}Under assumption \textbf{(H2) }the Niemytskij
operator $N_{f}:W_{0}^{1,p}(\Omega )\rightarrow W^{-1,p^{\prime }}(\Omega )$
induced by the Carath\'{e}odory function $f:\Omega \times \mathbb{R}\times 
\mathbb{R}^{N}\rightarrow \mathbb{R}$, namely%
\begin{equation*}
N_{f}(w)=f(\cdot ,w(\cdot ),\nabla w(\cdot )),\quad \forall w\in
W_{0}^{1,p}(\Omega ),
\end{equation*}%
is well defined, continuous and such that there exists a constant $C>0$ for
which the following estimate holds%
\begin{equation*}
\left\Vert N_{f}(u)\right\Vert _{W^{-1,p^{\prime }}(\Omega )}\leq C\left(
\left\Vert \sigma \right\Vert _{L^{r_{1}}}+\left\Vert u\right\Vert
_{L^{r_{2}}}^{r_{2}}+\left\Vert \nabla u\right\Vert _{L^{p}}^{p-1}\right)
,\quad \text{for all }u\in W_{0}^{1,p}(\Omega ).
\end{equation*}
\end{corollary}

Now we proceed to some bounds that are obtained on the solutions provided
they exist.

\begin{theorem}
\label{theorem_coercive}Under assumptions \textbf{(H1)}-\textbf{(H3)} there
exist constants $R,R_{1}>0$ such that if $u\in W_{0}^{1,p}(\Omega )$ is any
weak solution to (\ref{problemCL}) then $\left\Vert u\right\Vert
_{W_{0}^{1,p}\left(\Omega\right)}\leq R_{1}$ and $\left\Vert u\right\Vert _{C\left( \overline{%
\Omega }\right) }\leq R$.
\end{theorem}
\begin{proof}
When we insert $v=u$ in formula (\ref{def_operat_A}) we obtain what follows 
\begin{equation*}
\left\langle A\left( u\right) ,u\right\rangle =\int_{\Omega }g(
u\left( x\right)  )|\nabla u\left( x\right) |^{p}\mathrm{~d}%
x-\left\Vert \nabla u\right\Vert _{L^{q}\left(\Omega\right)}^{q}-\int_{\Omega }f\left(
x,u\left( x\right) ,\nabla u\left( x\right) \right) u\left( x\right) \mathrm{%
~d}x.
\end{equation*}%
From assumption \textbf{(H3)} and the Poincar\'e inequality we obtain the
following estimation 
\[
\begin{split}
\|&\nabla{u}\|^q_{L^q\left(\Omega\right)} + \int_{\Omega}f(x,u(x),\nabla{u}(x))\,\mathrm{d}x \\ 
& \leq |\Omega|^{\frac{p-q}{p}}\|\nabla{u}\|^q_{L^p\left(\Omega\right)} + c_0 \|\nabla{u}\|^p_{L^p\left(\Omega\right)} + c_1\int_{\Omega}|u(x)|^{\alpha}\,\mathrm{d}x + c_1|\Omega| \\ 
&\leq|\Omega|^{\frac{p-q}{p}}\|\nabla{u}\|^q_{L^p\left(\Omega\right)} + c_0 \|\nabla{u}\|^p_{L^p\left(\Omega\right)}  + c_1 |\Omega|^{\frac{p-\alpha}{p}}\|u\|^\alpha_{L^p\left(\Omega\right)} + c_1|\Omega| \\
&\leq|\Omega|^{\frac{p-q}{p}}\|\nabla{u}\|^q_{L^p\left(\Omega\right)} + c_0 \|\nabla{u}\|^p_{L^p\left(\Omega\right)} + \frac{c_1|\Omega|^{\frac{p-\alpha}{p}}}{\lambda^\alpha_1}\|\nabla{u}\|^\alpha_{L^p\left(\Omega\right)}
+c_1|\Omega|.
\end{split}
\]
Moreover%
\begin{equation*}
\int_{\Omega}g( u( x))|\nabla u\left(
x\right) |^{p}\mathrm{~d}x\geq a_{0}\left\Vert \nabla u\right\Vert
_{L^{p}(\Omega )}^{p}.
\end{equation*}%
Summarizing 
\begin{equation}
\left( a_{0}-c_{0}\right) \left\Vert \nabla u\right\Vert _{L^{p}(\Omega
)}^{p}\leq |\Omega |^{\frac{p-q}{p}}\left\Vert \nabla u\right\Vert
_{L^{p}(\Omega )}^{q}+\frac{c_1|\Omega|^{\frac{p-\alpha}{p}}}{\lambda^\alpha_1}\|\nabla{u}\|^\alpha_{L^p\left(\Omega\right)}
+c_1|\Omega|.  \label{estimeate_coercivity}
\end{equation}%
Since $q<p$ and $\alpha <p$ we see that there exists $R_{1}>0$
such that $\left\Vert u\right\Vert _{W_{0}^{1,p}}\leq R_{1}$. Then by
inequality (\ref{Sobolev Ineq}) we see that $\left\Vert u\right\Vert
_{C\left( \overline{\Omega }\right) }\leq R$ for $R=R_{1}\cdot C_{S}$.
\end{proof}

\section{On the truncated problem}

The next proposition focuses on the properties of the competitive $\left(
p,q\right) $-Laplacian associated to the truncated weight $g_{R}(
u\left( x\right) )$ with $R$ determined by Theorem \ref%
{theorem_coercive}. Now we consider the operator $A_{R}:W_{0}^{1,p}(\Omega
)\rightarrow W^{-1,p^{\prime }}(\Omega )$%
\begin{equation}
\begin{split}
\left\langle A_{R}\left( u\right) ,v\right\rangle =&\int_{\Omega
}g_{R}( u\left( x\right) )|\nabla u\left( x\right)
|^{p-2}\nabla u\left( x\right) \nabla w\left( x\right) \mathrm{~d}x \\
&-\int_{\Omega }|\nabla u|^{q-2}\nabla u\left( x\right) \nabla w\left(
x\right) \mathrm{~d}x-\int_{\Omega }f(x,u\left( x\right) ,\nabla u\left(
x\right) )w\left( x\right)\,\textup{d}x%
\end{split}
\label{def_A_R}
\end{equation}%
for $u,v\in W_{0}^{1,p}(\Omega )$, connected to problem (\ref{problemCL_aux}%
). For the sake of notation we will write 
\begin{equation*}
\left\langle A_{R}^{1}\left( u\right) ,v\right\rangle =\int_{\Omega
}g_{R}(u\left( x\right))|\nabla u\left( x\right)
|^{p-2}\nabla u\left( x\right) \nabla v\left( x\right) \mathrm{~d}x
\end{equation*}%
and 
\begin{equation*}
\left\langle A_{R}^{2}\left( u\right) ,v\right\rangle =\int_{\Omega }|\nabla
u(x)|^{q-2}\nabla u\left( x\right) \nabla v\left( x\right) \mathrm{~d}%
x+\int_{\Omega }f(x,u\left( x\right) ,\nabla u\left( x\right) )v\left(
x\right) \mathrm{~d}x
\end{equation*}%
for $u,v\in W_{0}^{1,p}(\Omega )$. Then of course $A_{R}=A_{R}^{1}+A_{R}^{2}$%
. Moreover, operator $A_{R}$ has the following properties (with any $R>0$):

\begin{proposition}
\label{proposition_prop_of_AR}Let $R>0$ be fixed. Assume that conditions 
\textbf{(H1)-(H3)} are satisfied. Then the following assertions hold:\newline
(i) $A_{R}$ is well defined and bounded (in the sense that it maps bounded
sets into bounded sets);\newline
(ii) $A_{R}^{1}$ has the $S_{+}$ property, that is, any sequence $\left\{
u_{n}\right\} \subset W_{0}^{1,p}(\Omega )$ with $u_{n}\rightharpoonup u$ in 
$W_{0}^{1,p}(\Omega )$ and%
\begin{equation*}
\limsup_{n\rightarrow \infty }\left\langle A_{R}\left( u_{n}\right)
,u_{n}-u\right\rangle \leq 0
\end{equation*}
\end{proposition}

satisfies $u_{n}\rightarrow u$ in $W_{0}^{1,p}(\Omega )$.\newline
(iii) $A_{R}$ is continuous.

\begin{proof}
We note that operator $A_{R}^{2}$ is obviously continuous. Moreover, it is
well defined and bounded due to Lemma \ref{lemmaestimationf} and due the
inequality 
\begin{equation*}
\left\Vert \nabla u\right\Vert _{L^{q}}^{q}\leq |\Omega |^{\frac{p-q}{p}}\left\Vert
\nabla u\right\Vert _{L^{p}(\Omega )}^{q}\text{ for }u\in W_{0}^{1,p}(\Omega
)\text{.}
\end{equation*}%
In view of the above remark we will show that $A_{R}^{1}$ satisfies the
conditions (i), (iii) above as well (ii). \newline
(i) By the continuity of $g$ we have for all $u,v\in W_{0}^{1,p}(\Omega )$
what follows:%
\[
\begin{split}
\int_{\Omega}g_R(u(x))|\nabla{u}(x)|^{p-1}|\nabla{v}(x)|\,\mathrm{d}x
&\leq \int_{\{x:|u(x)|\leq R\}}g_R(u(x))|\nabla{u}(x)|^{p-1}|\nabla{v}(x)|\,\mathrm{d}x \\
&+\int_{\{x:u(x)> R\}}g_R(R))|\nabla{u}(x)|^{p-1}|\nabla{v}(x)|\,\mathrm{d}x \\
&+\int_{\{x:u(x)< -R\}}g_R(-R)|\nabla{u}(x)|^{p-1}|\nabla{v}(x)|\,\mathrm{d}x\\
&\leq \max_{t\in[-R,R]}g(t)\int_{\{x:|u(x)|\leq R\}}|\nabla{u}(x)|^{p-1}|\nabla{v}(x)|\,\mathrm{d}x \\
&+ \max_{t\in[-R,R]}g(t)\int_{\{x:|u(x)|> R\}}|\nabla{u}(x)|^{p-1}|\nabla{v}(x)|\,\mathrm{d}x \\
& \leq \max_{t\in[-R,R]}g(t)\|u\|^{p-1}_{W^{1,p}_0\left(\Omega\right)}\|v\|_{W^{1,p}_0\left(\Omega\right)}.
\end{split}
\]
This means that $A_{R}^{1}$ is well defined and bounded.$\bigskip $\newline
(ii) Due to assumption that $p>N$ we can restrict ourselves to the case when 
$p>2$. Let $\left\{ u_{n}\right\} \subset W_{0}^{1,p}(\Omega )$ with $%
u_{n}\rightharpoonup u$ in $W_{0}^{1,p}(\Omega )$ and%
\begin{equation}
\limsup_{n\rightarrow \infty }\int_{\Omega }g_{R}\left( 
u_{n}\left( x\right)  \right) \left\vert \nabla
u_{n}(x)\right\vert ^{p-2}\left( \nabla u_{n}(x)-\nabla u(x)\right)\,\mathrm{d}x\leq 0.
\label{lim_sup_A1}
\end{equation}%
Using the inequality%
\begin{equation*}
\begin{split}
\left\langle -\Delta _{p}\left( u\right) -\left( -\Delta _{p}\left( v\right)
\right) ,u-v\right\rangle &= 
\int_{\Omega }\left( \left\vert \nabla u(x)\right\vert ^{p-2}\nabla
u(x)-\left\vert \nabla v(x)\right\vert ^{p-2}\nabla v(x)\right)\left(\nabla u(x)-\nabla
v(x)\right)\, \mathrm{d}x \\ & \geq   
\frac{1}{2^p}\int_{\Omega}\left\vert \nabla u(x)-\nabla
v(x)\right\vert ^{p}\,\mathrm{d}x
\end{split}
\end{equation*}%
for $u,v\in W_{0}^{1,p}\left( \Omega \right) $ we see that%
\begin{equation*}
\begin{split}
& \int_{\Omega }g_{R}\left(  u_{n}\left( x\right)
\right) \left\vert \nabla u_{n}(x)\right\vert ^{p-2}\left(
\nabla u_{n}(x)- \nabla u(x)\right)\,\mathrm{d}x\bigskip \\ 
& \geq a_{0}\int_{\Omega }\left( \left\vert \nabla u_{n}(x)\right\vert
^{p-2}\nabla u_{n}(x)-|\nabla u(x)|^{p-2}\nabla u(x)\right)  \left(\nabla
u_{n}(x)-\nabla u(x)\right) \mathrm{~d}x \\ 
& + \int_{\Omega}g_R\left(u_n(x)\right)|\nabla u(x)|^{p-2}\nabla u(x) \nabla(u_n - u)(x)\,\mathrm{d}x \\
& \geq \frac{a_0}{2^p}\int_{\Omega}\left\vert \nabla
u_{n}(x)-\nabla u(x)\right\vert ^{p}\,\mathrm{d}x + \int_{\Omega}g_R\left(u_n(x)\right)|\nabla u(x)|^{p-2}\nabla u(x) \nabla(u_n - u)(x)\,\mathrm{d}x.
\end{split}
\end{equation*}%
Hence using relation (\ref{lim_sup_A1}) we see that $u_{n}\rightarrow u$ in $%
W_{0}^{1,p}(\Omega )$.$\bigskip $\newline
(iii) Let $u_{n}\rightarrow u$ in $W_{0}^{1,p}(\Omega )$. Note that sequence 
$\left\{ u_{n}\right\} $ can be chosen so that it converges a.e. on $\Omega $%
. We see that for any $v\in W_{0}^{1,p}(\Omega )$ it holds 
\begin{equation*}
\begin{split}
\left\vert \left\langle A_{R}^{1}\left( u_{n}\right)
-A_{R}^{1}(u),v\right\rangle \right\vert \bigskip &\leq  
\max_{t\in \lbrack -R,R]}g(t)\bigskip \int_{\Omega }\left\vert\left(
\left\vert \nabla u_{n}(x)\right\vert ^{p-2}\nabla u_{n}(x)-\left\vert
\nabla u(x)\right\vert ^{p-2}\nabla u(x)\right) \nabla v(x)\right\vert\,\mathrm{d}x \\
& + \int_{\Omega }\left\vert g_{R}\left(u_{n}(x)
\right) -g_{R}(u(x))\right\vert \left\vert \nabla u(x)\right\vert
^{p-1}\left\vert \nabla v(x)\right\vert \,\mathrm{d}x.%
\end{split}
\end{equation*}%
We see that classical arguments provide 
\begin{equation*}
\left\vert \int_{\Omega }\left( \left\vert \nabla u_{n}(x)\right\vert
^{p-2}\nabla u_{n}(x)-\left\vert \nabla u(x)\right\vert ^{p-2}\nabla
u(x)\right) \nabla v(x)\,\mathrm{d}x\right\vert \rightarrow 0\text{ as }%
u_{n}\rightarrow u\text{.}
\end{equation*}%
Since $u_{n}\left( x\right) \rightarrow u\left( x\right) $ for a.e. $x\in
\Omega $ we see that by the continuity of $g_{R}$ that%
\begin{equation*}
 g_{R}\left( u_{n}(x) \right)
-g_{R}(u(x)) \rightarrow 0\text{ for a.e. }x\in \Omega .
\end{equation*}%
Since also 
\begin{equation*}
\left\vert g_{R}\left(  u_{n}(x) \right)
-g_{R}(u(x))\right\vert \leq 2\max_{t\in \lbrack -R,R]}g(t)
\end{equation*}%
we obtain that%
\begin{equation*}
\int_{\Omega }\left\vert g_{R}\left( u_{n}(x) \right)
-g_{R}(u(x))\right\vert \left\vert \nabla u(x)\right\vert ^{p-1}\left\vert
\nabla v(x)\right\vert \,\mathrm{d}x\rightarrow 0\text{ as }u_{n}\rightarrow u
\end{equation*}%
as well. This provides the continuity of $A_{R}^{1}$.
\end{proof}

\begin{remark}
We note here that operator $A_{R}$ does not satisfy property $S_{+}$.
Indeed, if it did we would arrive at the weak continuity of the $%
W_{0}^{1,p}(\Omega )$ norm which is not true. The reasoning is similar to
this given in \cite{MotreanuOpenMath} and it will not be repeated here.
\end{remark}

Concerning relation between problems (\ref{problemCL}) and (\ref%
{problemCL_aux}) we see that by Theorem \ref{theorem_coercive} the
estimations obtained are independent of the solutions and therefore both
problems coincide. We provide this observation in the following:

\begin{theorem}
Under assumptions \textbf{(H1)}-\textbf{(H3)} weak solutions to (\ref%
{problemCL}) and (\ref{problemCL_aux}) coincide.
\end{theorem}

\section{On the solvability of the truncation problem}

Concerning definition of a generalized solution we follow \cite{MotreanuOpenMath}.

\begin{definition}
\label{definition_generalized sol}Assume that hypothesis \textbf{(H1)-(H3)}
are verified. A function $u\in W_{0}^{1,p}(\Omega )$ is said to be a
generalized solution to problem (\ref{problemCL_aux}) if there exists a
sequence $\left\{ u_{n}\right\} _{n\geq 1}$ in $W_{0}^{1,p}(\Omega )$ such
that\newline
(a) $u_{n}\rightharpoonup u$ in $W_{0}^{1,p}(\Omega )$ as $n\rightarrow
\infty $;\newline
(b) $\lim_{n\rightarrow \infty }\left\langle A_{R}\left( u_{n}\right)
,v\right\rangle =0$ for each $v\in W_{0}^{1,p}(\Omega )$; \newline
(c) $\lim_{n\rightarrow \infty }\left\langle A_{R}\left( u_{n}\right)
,u_{n}-u\right\rangle =0$.
\end{definition}

\begin{remark}
Writing conditions (b) and (c) explicitly we see that \newline
(b) $-\dive\left( g_{R}( u_n\left( \cdot \right) 
)|\nabla u_n\left( \cdot \right) |^{p-2}\nabla u_n\left( \cdot \right) \right)
+\Delta _{q}u_{n}\left( \cdot \right) -f\left( \cdot ,u_{n}(\cdot ),\nabla
u_{n}(\cdot )\right) \rightharpoonup 0$ in $W^{-1,p^{\prime }}(\Omega )$ as $%
n\rightarrow \infty $;\newline
and\newline
(c) $\lim_{n\rightarrow \infty }\left[ \int_{\Omega }\left( g_{R}(
u_n\left( x\right) )|\nabla u_n|^{p-2}\nabla u_n-|\nabla
u_n|^{q-2}\nabla u_n\right) \nabla \left( u_{n}-u\right) \mathrm{~d}%
x-\int_{\Omega }f\left( x,u_{n},\nabla u_{n}\right) \left( u_{n}-u\right) 
\mathrm{d}x\right] =0$.
\end{remark}

With Proposition \ref{proposition_prop_of_AR} we see that the above
definition is well posed. We provide some additional definition as well:

\begin{definition}
\label{definition_strong_generalized_sol}A function $u\in W_{0}^{1,p}(\Omega
)$ is said to be a strong generalized solution to problem (\ref%
{problemCL_aux}) if there exists a sequence $\left\{ u_{n}\right\} _{n\geq
1} $ in $W_{0}^{1,p}(\Omega )$ such that (a) and (b) in Definition \ref%
{definition_generalized sol} are satisfied together with (c)'%
\begin{equation*}
\lim_{n\rightarrow \infty }\left\langle A_{R}^{1}\left( u_{n}\right) +\Delta
_{q}u_{n},u_{n}-u\right\rangle =0.
\end{equation*}
\end{definition}

Since the Banach space $W_{0}^{1,p}(\Omega )$ with $1<p<+\infty $ is
separable, we can fix a Galerkin basis of $W_{0}^{1,p}(\Omega )$, that is a
sequence $\left\{ E_{n}\right\} _{n\geq 1}$ of finite dimensional vector
subspaces of $W_{0}^{1,p}(\Omega )$ satisfying\newline
(i) $\dim\left( E_{n}\right) <\infty ,\quad \forall n$;\newline
(ii) $E_{n}\subset E_{n+1},\quad \forall n$;\newline
(iii) $\bigcup_{n=1}^{\infty }E_{n}=W_{0}^{1,p}(\Omega )$.
  
\begin{proposition}
\label{Proposition_exist_galerkin_scheme}Assume that conditions \textbf{(H2)}%
-\textbf{(H3)} are fulfilled. Then for each $n\geq 1$ there exists $u_{n}\in
E_{n}$ such that%
\begin{equation}
\left\langle A_{n,R}\left( u_{n}\right) ,v\right\rangle =0\quad \text{ for
all }v\in E_{n}.  \label{equ_aux_Galer}
\end{equation}%
Moreover, the sequence $\left\{ u_{n}\right\} _{n\geq 1}$, with $u_{n}\in E_{n}$, is
bounded in $W_{0}^{1,p}(\Omega )$.
\end{proposition}

\begin{proof}
For each $n\geq 1$ by $A_{n,R}$ we understand the restriction of $A_{R}$ to the space $E_{n}$, i.e. 
$A_{n,R}:E_{n}\rightarrow E^{*}$.
From relation (\ref{estimeate_coercivity}) we now obtain%
\begin{equation}
\left( a_{0}-c_{0}\right) \left\Vert \nabla v\right\Vert _{L^{p}(\Omega
)}^{p}\leq |\Omega |^{\frac{p-q}{p}}\left\Vert \nabla v\right\Vert
_{L^{p}(\Omega )}^{q}+\frac{c_1|\Omega|^{\frac{p-\alpha}{p}}}{\lambda^\alpha_1}\|\nabla{v}\|^\alpha_{L^p\left(\Omega\right)}
+c_1|\Omega| \text{ for }v\in E_{n}.
\label{estim_frorm_galerkin_lemma}
\end{equation}
Using that $p>q,p>\alpha $ and $c_{0}<a_0$
we conclude that if $R>0$ which is independent of $n$ is sufficiently large
then%
\begin{equation*}
\left\langle A_{n,R}(v),v\right\rangle \geq 0\text{ whenever }v\in E_{n}\text{
with }\Vert \nabla v\Vert _{L^{p}(\Omega )}=R.
\end{equation*}%
Using Lemma \ref{lemma from Brouwer} with $X=E_{n}$ and $A=A_{R,n}$ we see
that there is $u_{n}\in E_{n}$ solving (\ref{equ_aux_Galer}). Since such $%
u_{n}$ satisfies relation (\ref{estim_frorm_galerkin_lemma}) we obtain the
second assertion as well.
\end{proof}

We now proceed to the existence result:

\begin{theorem}
\label{theorem_generazlied_solution}Assume that conditions \textbf{(H1)}-%
\textbf{(H3)} hold. Then there exists a generalized solution to problem (\ref%
{problemCL_aux}).
\end{theorem}

\begin{proof}
We shall show that with this sequence using Definition \ref%
{definition_generalized sol} we obtain the generalized solution to (\ref%
{problemCL_aux}). The sequence $\left\{ u_{n}\right\} _{n\geq 1}$ obtained
in Proposition \ref{Proposition_exist_galerkin_scheme} is bounded in $%
W_{0}^{1,p}(\Omega )$\ and thus contains a weakly convergent subsequence
(which we do not renumber) with a weak limit $u\in W_{0}^{1,p}(\Omega )$.
Thus we have condition (a) of the definition mentioned satisfied. Using
Proposition \ref{proposition_prop_of_AR} we arrive at conclusion that the
sequence required in condition (b) is weakly convergent (possibly up to a
subsequence which we assume to be chosen and which we again do not
renumber). The weak limit obtained we denote by $\eta \in W^{-1,p^{\prime
}}(\Omega )$.

We show that $\eta =0$. Let $v\in \bigcup_{n\geq 1}E_{n}$. There is an
integer $m\geq 1$ such that $v\in E_{m}$. Applying Proposition \ref%
{proposition_prop_of_AR}, we see that equality (\ref{equ_aux_Galer}) holds
true for all $n\geq m$. Letting $n\rightarrow \infty $ in (\ref%
{equ_aux_Galer}) we obtain what follows%
\begin{equation*}
\langle \eta ,v\rangle =0\text{ for all }v\in \bigcup_{n\geq 1}E_{n}.
\end{equation*}%
Since $\bigcup_{n\geq 1}E_{n}$ is dense in $W_{0}^{1,p}(\Omega )$ we see
that $\eta =0$. This provides condition (b).

Now we turn to condition (c). Since $\eta =0$ we see at once that 
\begin{equation}
\lim_{n\rightarrow \infty }\left\langle A_{R}\left( u_{n}\right)
,u\right\rangle =0.  \label{limit with u}
\end{equation}%
Putting $v=u_{n}$ in (\ref{equ_aux_Galer}), which provides $\left\langle
A_{R}\left( u_{n}\right) ,u_{n}\right\rangle =0$ for all $n\in 
\mathbb{N}
$ which obviously implies that%
\begin{equation}
\lim_{n\rightarrow \infty }\left\langle A_{R}\left( u_{n}\right)
,u_{n}\right\rangle =0.  \label{limit_with_u_n}
\end{equation}%
Now combining (\ref{limit with u}) and (\ref{limit_with_u_n}) provides that
condition (c) is also satisfied and therefore $u$ is the generalized
solution to (\ref{problemCL_aux}).
\end{proof}

Now we proceed to the existence of strong generalized solution:

\textbf{(H4)} \textit{There exist constants }$c_{1}\geq 0,c_{2}\geq 0,r\in %
\left[ 1,p^{\star }\right) ,r_{1}\in \left[ 1,p^{\star }\right) ,r_{2}\in
\lbrack 1,p)$\textit{\ and a nonnegative function }$\sigma \in L^{r^{\prime
}}(\Omega )$\textit{\ such that}

\begin{equation*}
|f(x,s,\xi )|\leq \sigma (x)+c_{1}|s|^{\frac{p^{\ast }}{r_{1}^{^{\prime }}}%
}+c_{2}|\xi |^{\frac{p}{r_{2}^{^{\prime }}}}\quad \text{ \textit{for a.e.} }%
x\in \Omega \text{, \textit{all} }s\in 
\mathbb{R}
,\xi \in 
\mathbb{R}
^{N}\text{. }
\end{equation*}%
The following lemma is taken after \cite{MotreanuOpenMath}:

\begin{lemma}
\label{convergence lemma}Assume that condition\textbf{\ (H4)} holds. The for
any sequence $\left\{ u_{n}\right\} _{n\geq 1}\subset W_{0}^{1,p}(\Omega )$
such that $u_{n}\rightharpoonup u$ in $W_{0}^{1,p}(\Omega )$ it holds:%
\begin{equation*}
\lim_{n\rightarrow \infty }\int_{\Omega }f\left( x,u_{n},\nabla u_{n}\right)
\left( u_{n}-u\right) \mathrm{d}x=0.
\end{equation*}
\end{lemma}

\begin{theorem}
Assume that conditions \textbf{(H1), (H3), (H4)} hold. Then there exists a
strong generalized solution to problem (\ref{problemCL_aux}).
\end{theorem}

\begin{proof}
As mentioned in \cite{MotreanuOpenMath} Remark 2.4. condition \textbf{(H4)}
implies condition \textbf{(H2)}. Thus by Theorem \ref%
{theorem_generazlied_solution} we obtain the existence of the generalized
solution with our set of assumption. We need to show that 
\begin{equation}
\lim_{n\rightarrow \infty }\left\langle A_{R}^{1}\left( u_{n}\right) +\Delta
_{q}u_{n},u_{n}-u\right\rangle =0.  \label{assert from strong generalized}
\end{equation}%
From Definition \ref{definition_generalized sol} condition (c) we know that 
\begin{equation*}
\lim_{n\rightarrow \infty }\left\langle A_{R}\left( u_{n}\right)
,u_{n}-u\right\rangle =0
\end{equation*}%
which by Lemma \ref{convergence lemma} implies (\ref{assert from strong
generalized}).
\end{proof}

\section{Abstract result}

Now we turn to writing an abstract result concerning the existence of a
generalized solution although not the strong generalized solution. Assume
that $E$ is a separable reflexive Banach space and let $A:E\rightarrow
E^{\ast }$. For any fixed $f\in E^{\ast }$ we consider the following problem%
\begin{equation}
A\left( u\right) =f  \label{abstract equation}.
\end{equation}

\begin{definition}
\label{definition-abstract generalized solution}An element $u\in E$ is said
to be a generalized solution to problem (\ref{abstract equation}) if there
exists a sequence $\left\{ u_{n}\right\} _{n\geq 1}$ in $E$ such that\newline
(a) $u_{n}\rightharpoonup u$ in $E$ as $n\rightarrow \infty $;\newline
(b) $\lim_{n\rightarrow \infty }\left\langle A\left( u_{n}\right)
-f,v\right\rangle =0$ for each $v\in E$; \newline
(c) $\lim_{n\rightarrow \infty }\left\langle A\left( u_{n}\right)
-f,u_{n}-u\right\rangle =0.$
\end{definition}

We recall that operator $A:E\rightarrow E^{\ast }$ satisfies condition (S),
or we say that $A$ has property (S), if relations 
\begin{equation*}
u_{n}\rightharpoonup u_{0}\text{ in }E
\end{equation*}%
and 
\begin{equation*}
\left\langle A(u_{n}),u_{n}-u_{0}\right\rangle \rightarrow 0
\end{equation*}%
imply that 
\begin{equation*}
u_{n}\rightarrow u_{0}\text{ in }E.
\end{equation*}

\begin{remark}
\label{remark about apprpximatopn space}Since $E$ is separable it contains a
dense and countable set $\left\{ h_{1},...,h_{n},...\right\} $. Define $%
E_{n} $ for $n\in 
\mathbb{N}
$ as a linear hull of $\left\{ h_{1},...,h_{n}\right\} $. The sequence of
subspaces $E_{n}$ has the approximation property: for each $u\in E$ there is
a sequence $\left( u_{n}\right) _{n=1}^{\infty }$ such that $u_{n}\in E_{n}$
for $n\in 
\mathbb{N}
$ and $u_{n}\rightarrow u$ and moreover, $E_{n}\subset E_{n+1}$ for $n\in 
\mathbb{N}
$\ and $\overline{\bigcup_{n=1}^{\infty }E_{n}}=E.$
\end{remark}

\begin{theorem}
\label{theorem - abstract existence}Assume that $A:E\rightarrow E^{\ast }$ a
continuous, coercive and bounded operator. Then problem (\ref{abstract
equation}) has at least one generalized solution. If additionally $A$
satisfies the condition (S) then, any generalized solution is a weak
solution.
\end{theorem}

\begin{proof}
Let us fix $n\in 
\mathbb{N}
$ and space $E_{n}$ from Remark \ref{remark about apprpximatopn space}. By $%
f_{n}$ we denote the restriction of functional $f$ to space $E_{n}$.
Similarly by $A_{n}$ we understand the restriction of $A$ to space $E_{n}$.
Then%
\begin{equation*}
A_{n}:E_{n}\rightarrow E_{n}^{\ast }
\end{equation*}%
is continuous and coercive. Due to the coercivity of $A$ we see that there
is some $R>0$ which is independent of $n$ is sufficiently large and such that%
\begin{equation*}
\left\langle A_{n}\left( v\right) ,v\right\rangle \geq 0\text{ whenever }%
v\in E_{n}\text{ with }\Vert v\Vert _{E}=R.
\end{equation*}%
By Lemma \ref{lemma from Brouwer} we now see that equation%
\begin{equation}
A_{n}\left( u\right) =f_{n}  \label{Anu=fn}
\end{equation}%
has at least one solution $u_{n}$ with $\Vert u_{n}\Vert _{E}\leq R$. Since $%
\left( u_{n}\right) _{n=1}^{\infty }$ is bounded, there is a subsequence,
which we do not renumber, convergent weakly to some $u$. Thus we have
condition (a) from Definition \ref{definition-abstract generalized solution}
satisfied. Since operator $A$ is bounded it follows\ that $\left\Vert
A_{n}\left( u_{n}\right) \right\Vert _{\ast }\leq M_{1}$ for some fixed $%
M_{1}>0$ and for all $n\in 
\mathbb{N}
$. Hence it follows from (\ref{Anu=fn}), again possibly for a subsequence,
that 
\begin{equation*}
\lim_{n\rightarrow +\infty }\left\langle A_{n}\left( u_{n}\right)
,h\right\rangle =\left\langle f,h\right\rangle \text{ for each }h\in
\bigcup_{n=1}^{\infty }E_{n}\text{.}
\end{equation*}%
By density and since $A_{n}\left( u_{n}\right) =A\left( u_{n}\right) $ we
obtain that condition (b) from Definition \ref{definition-abstract
generalized solution} is satisfied. This also means that $A\left(
u_{n}\right) \rightharpoonup f$. Moreover, testing (\ref{Anu=fn}) against $%
u_{n}$ we see that 
\begin{equation*}
\left\langle A\left( u_{n}\right) ,u_{n}\right\rangle =\left\langle
f,u_{n}\right\rangle \text{.}
\end{equation*}%
Since $\lim_{n\rightarrow +\infty }\left\langle A\left( u_{n}\right)
-f,u\right\rangle =0$ we observe that condition (c) from Definition \ref%
{definition-abstract generalized solution} also holds.

Now assume that $A$ satisfies the condition (S). Now we see that from $%
u_{n}\rightharpoonup u$ and from condition (c) it follows that $%
u_{n}\rightarrow u$ in $E$. From definition of $A_{n}$ and from (\ref{Anu=fn}%
) we see that for all $n\in 
\mathbb{N}
$ 
\begin{equation*}
A\left( u_{n}\right) =f_n.
\end{equation*}%
By the continuity of $A$ we now obtain that $A\left( u\right) =f$ and the
assertion follows.
\end{proof}

From the above, we see that a generalized solution becomes a weak one on
assumption of condition (S) or other related compactness condition. These
solutions coincide. Now we obtain the existence of a weak solution to (\ref%
{abstract equation}) from the existence of the generalized solution.

\begin{proposition}
\label{proposition - weak implies generalized}Assume that $u\in E$ is a weak
solution to (\ref{abstract equation}). Then it is also a generalized
solution.
\end{proposition}

\begin{proof}
In order to reach the conclusion we see that if $u\in E$ is a weak solution
to (\ref{abstract equation}) we put $\left\{ u_{n}\right\} _{n\geq
1}=\left\{ u\right\} _{n\geq 1}$ and observe that conditions (a), (b), (c)
from Definition \ref{definition-abstract generalized solution}\ are now
satisfied.
\end{proof}

\section{Some applications of abstract result}

Due to Proposition \ref{proposition_prop_of_AR} and Theorem \ref%
{theorem_coercive} we see that Theorem \ref{theorem - abstract existence}
implies Theorem \ref{theorem_generazlied_solution}. While from Proposition %
\ref{proposition - weak implies generalized} we see that

\begin{lemma}
Assume that conditions \textbf{(H1)-(H3)} are satisfied. Assume that $u\in
W_{0}^{1,p}(\Omega )$ is a weak solution to (\ref{problemCL_aux}). Then it
is also a generalized solution.
\end{lemma}

\begin{remark}
We note that we cannot argue that even a strong generalized solution to (\ref%
{problemCL_aux}) is its weak solution. Indeed, if $u$ is a strong
generalized solution we arrive at the following\newline
i) $u_{n}\rightharpoonup u$ in $W_{0}^{1,p}(\Omega )$ as $n\rightarrow
\infty $;\newline
ii) $\lim_{n\rightarrow \infty }\left[ \int_{\Omega }\left( g_{R}(
u\left( x\right) )|\nabla u|^{p-2}\nabla u-|\nabla
u|^{q-2}\nabla u\right) \nabla \left( u_{n}-u\right) \mathrm{~d}x\right] =0$.%
\newline
Hence for $u$ to be a weak solution, it would mean that $u_{n}\rightarrow u$
which is impossible.
\end{remark}

\begin{remark}
Moreover, due to the type of a definition of a strong generalized solution
it seems that for the time being there is no abstract counterpart of it.
\end{remark}

Finally we apply Theorem \ref{theorem - abstract existence} in examining the
solvability of the following problem%
\begin{equation}
\begin{cases}
-\dive\left( g( u\left( x\right) )|\nabla u\left(
x\right) |^{p-2}\nabla u\left( x\right) +|\nabla u\left( x\right)
|^{q-2}\nabla u\left( x\right) \right) =f(x,u\left( x\right) ,\nabla u\left(
x\right) ) & \text{ in }\Omega , \\ 
u\left( x\right) =0 & \text{ on }\partial \Omega 
\end{cases}
\label{problemDF}
\end{equation}%
under assumptions \textbf{(H1), (H3), (H4). }Recall that\textbf{\ (H4) }%
implies\textbf{\ (H2). }Note that due to assumption\textbf{\ (H1) }problem%
\textbf{\ }(\ref{problemDF}) is of independent interest since it cannot be
treated directly. For this problem we obtain that the generalized, the
strong generalized and the weak solutions coincide. With almost the same
proof as for Theorem \ref{theorem_coercive} we obtain:

\begin{theorem}
\label{Theo_EstimDF}Assume that conditions \textbf{(H1), (H3), (H4)}. Then
there exists a constant $R>0$ such that for each weak solution $u\in
W_{0}^{1,p}(\Omega )$ to problem (\ref{problemDF}) it holds the uniform
estimate $\Vert u\Vert _{C(\overline{\Omega })}\leq R$. The constant $R$
depends on $g$ only through its lower bound $a_{0}$.
\end{theorem}

Further we consider a truncated counterpart of (\ref{problemDF}) with $g_{R}$
defined by (\ref{def of g}): 
\begin{equation}
\begin{cases}
-\dive\left( g_{R}( u\left( x\right) )|\nabla
u\left( x\right) |^{p-2}\nabla u\left( x\right) +|\nabla u\left( x\right)
|^{q-2}\nabla u\left( x\right) \right) =f(x,u\left( x\right) ,\nabla u\left(
x\right) ) & \text{ in }\Omega , \\ 
u\left( x\right) =0 & \text{ on }\partial \Omega 
\end{cases}
\label{problemDFG}
\end{equation}%
We define $A_{R}^{1}$, $A_{R}^{2}:W_{0}^{1,p}(\Omega )\rightarrow
W^{-1,p^{\prime }}(\Omega )$ 
\begin{equation*}
\left\langle A_{R}^{1}\left( u\right) ,v\right\rangle =\int_{\Omega
}g_{R}(u\left( x\right) )|\nabla u\left( x\right)
|^{p-2}\nabla u\left( x\right) \nabla v\left( x\right) \mathrm{~d}%
x+\int_{\Omega }|\nabla u(x)|^{q-2}\nabla u\left( x\right) \nabla v\left(
x\right) \mathrm{~d}x
\end{equation*}%
and 
\begin{equation*}
\left\langle A_{R}^{2}\left( u\right) ,v\right\rangle =\int_{\Omega
}f(x,u\left( x\right) ,\nabla u\left( x\right) )w\left( x\right) \mathrm{~d}x
\end{equation*}%
for $u,v\in W_{0}^{1,p}(\Omega )$. Then we put $\widetilde{A_{R}}%
=A_{R}^{1}+A_{R}^{2}$. We see that operator $A_{R}$ has the following
properties:

\begin{proposition}
\label{proposition_prop_of_AR copy(1)}Let $R>0$ be fixed. Assume that
conditions \textbf{(H1), (H3), (H4)} are satisfied. Then the following
assertions hold:\newline
(i) $\widetilde{A_{R}}$ is well defined and bounded (in the sense that it
maps bounded sets into bounded sets);\newline
(ii) $\widetilde{A_{R}}$ has the $S_{+}$ property, that is, any sequence $%
\left\{ u_{n}\right\} \subset W_{0}^{1,p}(\Omega )$ with $%
u_{n}\rightharpoonup u$ in $W_{0}^{1,p}(\Omega )$ and%
\begin{equation*}
\limsup_{n\rightarrow \infty }\left\langle \widetilde{A_{R}}\left(
u_{n}\right) ,u_{n}-u\right\rangle \leq 0
\end{equation*}%
satisfies $u_{n}\rightarrow u$ in $W_{0}^{1,p}(\Omega )$.\newline
(iii) $\widetilde{A_{R}}$ is continuous.
\end{proposition}

\begin{proof}
Conditions (i) and (iii) are immediate from Proposition \ref%
{proposition_prop_of_AR}. Taking $\left\{ u_{n}\right\} \subset
W_{0}^{1,p}(\Omega )$ with $u_{n}\rightharpoonup u$ in $W_{0}^{1,p}(\Omega )$
and such that%
\begin{equation*}
\limsup_{n\rightarrow \infty }\left\langle \widetilde{A_{R}}\left(
u_{n}\right) ,u_{n}-u\right\rangle \leq 0
\end{equation*}%
we see by Lemma \ref{convergence lemma} that 
\begin{equation*}
\limsup_{n\rightarrow \infty }\left\langle A_{R}^{1}\left( u_{n}\right)
,u_{n}-u\right\rangle \leq 0.
\end{equation*}%
Hence by 
\[ 
\begin{split}
\left\langle A_{R}^{1}\left( u_{n}\right) ,u_{n}-u\right\rangle \bigskip  &= 
\int_{\Omega}g_R\left(u_n(x)\right)|\nabla u_{n}(x)|^{p-2}\nabla u_{n}(x)\nabla(u_n-u)(x)\,\mathrm{d}x 
+ \int_{\Omega}|\nabla u_{n}(x)|^{q-2}\nabla u_{n}(x)\nabla(u_n-u)(x)\,\mathrm{d}x \\
& \geq 
a_0 \int_{\Omega}\left(|\nabla u_{n}(x)|^{p-2}\nabla u_{n}(x)-|\nabla u(x)|^{p-2}\nabla u(x)\right)\nabla(u_n-u)(x)\,\mathrm{d}x \\  
& + \int_{\Omega}g_R\left(u_n(x)\right)|\nabla u(x)|^{p-2}\nabla u(x)\nabla(u_n-u)(x)\,\mathrm{d}x \\
& + \int_{\Omega}\left(|\nabla u_{n}(x)|^{q-2}\nabla u_{n}(x)-|\nabla u(x)|^{q-2}\nabla u(x)\right)\nabla(u_n-u)(x)\,\mathrm{d}x \\
& + \int_{\Omega}|\nabla u(x)|^{q-2}\nabla u(x)\nabla(u_n-u)(x)\,\mathrm{d}x
\\
& \geq \frac{a_0}{2^p}\int_{\Omega}\left\vert \nabla
u_{n}(x)-\nabla u(x)\right\vert ^{p}\,\mathrm{d}x + \int_{\Omega}g_R\left(u_n(x)\right)|\nabla u(x)|^{p-2}\nabla u(x)\nabla(u_n-u)(x)\,\mathrm{d}x \\
& +\frac{1}{2^q}\int_{\Omega}\left\vert \nabla
u_{n}(x)-\nabla u(x)\right\vert ^{q}\,\mathrm{d}x+ \int_{\Omega}|\nabla u(x)|^{q-2}\nabla u(x)\nabla(u_n-u)(x)\,\mathrm{d}x \\
& \geq \frac{a_0}{2^p}\int_{\Omega}\left\vert \nabla
u_{n}(x)-\nabla u(x)\right\vert ^{p}\,\mathrm{d}x + \int_{\Omega}g_R\left(u_n(x)\right)|\nabla u(x)|^{p-2}\nabla u(x)\nabla(u_n-u)(x)\,\mathrm{d}x \\
& + \int_{\Omega}|\nabla u(x)|^{q-2}\nabla u(x)\nabla(u_n-u)(x)\,\mathrm{d}x
\end{split}
\]
estimation we have that $u_{n}\rightarrow u$ in $W_{0}^{1,p}(\Omega )$.
\end{proof}

Now we proceed to the solvability of (\ref{problemDF}).

\begin{theorem}
Assume that conditions \textbf{(H1), (H3), (H4)} are satisfied. Then problem
(\ref{problemDF}) has at least one bounded solution.
\end{theorem}

\begin{proof}
Since condition (S$_{+}$) implies condition (S) by Proposition \ref%
{proposition_prop_of_AR copy(1)} we see that Theorem \ref{theorem - abstract
existence} can be applied to problem (\ref{problemDFG}) which has at least
one weak solution. Since any solution to (\ref{problemDFG}) solves also (\ref%
{problemDF}) we arrive at the conclusion.
\end{proof}

Finally, we return to our starting point, i.e. problem with relaxed
assumption \textbf{(H3)}. Namely we assume that 

\textbf{(H3a)} There exist constants $c_{0}<a_{0},$\ $c_{1}/\left( \lambda
_{1}\right) ^{p}<a_{0}-c_{0}$\ such that 

\begin{equation*}
f(x,s,\xi )s\leq c_{0}|\xi |^{p}+c_{1}\left( |s|^{p}+1\right) \text{ \textit{%
for a.e.} }x\in \Omega \text{, all }s\in \mathbb{R},\xi \in \mathbb{R}^{N}%
\text{. }
\end{equation*}%
The function $f:\Omega \times \mathbb{R}\times \mathbb{R}^{N}\rightarrow \mathbb{R}$\ given by%
\begin{equation*}
f(x,s,\xi )=|s|^{p -2}s+\frac{s}{1+s^{2}}\left( |\xi
|^{p-1}+h(x)\right) \text{ \textit{for all} }(x,s,\xi )\in \Omega \times 
\mathbb{R}\times \mathbb{R}^{N},
\end{equation*}%
for some $h\in L^{\infty }(\Omega )$\ satisfies conditions \textbf{(H2)}, \textbf{(H3a)}. We
see that Theorem \ref{theorem_coercive} holds under assumptions \textbf{(H1)}, \textbf{(H2)},
\textbf{(H3a)}. Thus we immediately obtain

\begin{theorem}
Assume that conditions \textbf{(H1)}, \textbf{(H2)}, \textbf{(H3a)} hold. Then there exists a
generalized solution to problem (\ref{problemCL_aux}).
\end{theorem}

\textbf{Conflict of interests.} The Authors report no confict of interests.

\textbf{Data availability.} No data have been used in this work.

\textbf{Acknowledgement.} The work of J. Dibl\'{\i}k was
supported by the project of specific university research FAST-S-22-7867 and
FEKT-S-23-8179, Brno University of Technology.

\begin{tabular}{l}
Josef Dibl\'{\i}k \\ 
Faculty of Electrical Engineering and Communication, \\ 
Department of Mathematics, \\ 
Technick\'{a} 3058/10, 616 00 Brno \\ 
and \\ 
Faculty of Civil Engineering, \\ 
Institute of Mathematics and Descriptive Geometry, \\ 
Veve\v{r}\'{\i} 331/95, 602 00 Brno, \\ 
\texttt{diblik@vut.cz}%
\end{tabular}
\begin{tabular}{l}
Marek Galewski \\ 
Institute of Mathematics, \\ 
Lodz University of Technology, \\ 
al. Politechniki 8, \\ 
93-590 Lodz, Poland \\ 
\texttt{marek.galewski@p.lodz.pl}%
\end{tabular}%

\begin{tabular}{l}
\\
Igor Kossowski \\ 
Institute of Mathematics, \\ 
Lodz University of Technology, \\ 
al. Politechniki 8, \\ 
93-590 Lodz, Poland \\ 
\texttt{igor.kossowski@p.lodz.pl}%
\end{tabular}%
\begin{tabular}{l}
Dumitru Motreanu \\ 
Department of Mathematics, \\ 
University of Perpignan, \\ 
66860 Perpignan, France \\ 
\texttt{motreanu@univ-perp.fr}%
\end{tabular}

\end{document}